\documentclass{amsart}

\usepackage{amscd,amssymb,epsfig, epsf,pinlabel,graphicx}
\usepackage{epic,eepic}
\usepackage[dvipsnames]{xcolor}

\usepackage[all]{xy}
\SelectTips{cm}{}
\allowdisplaybreaks

\numberwithin{equation}{subsection}

\newtheorem{propo}{Proposition}[section]
\newtheorem{corol}[propo]{Corollary}
\newtheorem{theor}[propo]{Theorem}
\newtheorem{lemma}[propo]{Lemma}

\theoremstyle{definition}

\theoremstyle{remark}

\let\oldmarginpar\marginpar
\renewcommand\marginpar[1]{\oldmarginpar{\footnotesize #1}}

\newcommand{\CC}{\mathbb{C}}

\newcommand{\RR}{\mathbb{R}}

\newcommand{\Hom}{\operatorname{Hom}}
\newcommand{\Ker}{\operatorname{Ker}}

\newcommand{\Int}{\operatorname{Int}}

\begin{document}

\title[Quasi-Poisson structures on moduli spaces of quasi-surfaces]{Quasi-Poisson structures on moduli spaces of quasi-surfaces}

   \author[Vladimir Turaev]{Vladimir Turaev}
    \address{
    Vladimir Turaev \newline
    \indent   Department of Mathematics \newline
    \indent  Indiana University \newline
    \indent Bloomington IN47405, USA\newline
    \indent and  \newline
    \indent IRMA, Strasbourg \newline
    \indent 7 rue Rene Descartes \newline
    \indent 67084 Strasbourg, France \newline
    \indent $\mathtt{vturaev@yahoo.com}$}

\begin{abstract} In generalization of the  classical Atiyah-Bott Poisson brackets   on the moduli spaces of  surfaces  we define  quasi-Poisson brackets   on the  moduli spaces of   quasi-surfaces.  
\end{abstract}

\maketitle

\section {Introduction}

Moduli spaces of  surfaces    carry beautiful geometric structures,   in particular,  Poisson brackets, see  \cite{AB}.  Our   goal  is to extend these   brackets to the more general setting
of quasi-surfaces   introduced in~\cite{Tu3}.  It turns out  that  the appropriate  version of the Jacobi identity    involves  both  a 2-bracket and   a  3-bracket. Pairs  of brackets satisfying this modified Jacobi identity are said to be quasi-Poisson.        Our main result    is a  construction of  quasi-Poisson pairs of brackets  on   the moduli spaces of   quasi-surfaces.

  A part of this work  concerns an arbitrary  algebra~$A$.  Following \cite{VdB},  \cite{Cb},  for any   integer $n\geq 1$,  we consider a  trace  algebra   $A^t_n  $ which, under appropriate   assumptions on~$A$  and the ground ring,    is the coordinate  algebra of  the affine quotient scheme ${\text {Rep}}_n(A)//GL_n$. We  view this affine scheme  as  the  moduli space  of $n$-dimensional representations of~$A$ and we view the trace  algebra $A^t_n$ as the algebra of functions on this space.  In generalization of the work of    Crawley-Boevey \cite{Cb},  we  show  how to derive  brackets in   
the trace algebras    from so-called  braces in~$A$.    We  formulate conditions on  the  braces  ensuring  that    the induced  brackets in   
the trace algebras form  quasi-Poisson pairs.

Quasi-surfaces are topological spaces generalizing surfaces with boundary by allowing singular (non-surface) parts. A number of homotopy-invariant operations  on loops in surfaces generalize to quasi-surfaces, see \cite{Tu3} and Sections~\ref{A topological example}, \ref{Quasi-surfaces and braces} below. We  show that   these  operations are  braces in   the group algebra~$A$   of  the fundamental group.  We then show that the induced   brackets in   
the trace algebras $\{A^t_n\}_{n\geq 1}$ form  quasi-Poisson pairs.

The first part of the paper (Sections \ref{Preliminaries},~\ref{AMdddfT0}) presents our  algebraic methods and the second part  (Sections \ref{A topological example2}--\ref{Quasi-surfaces and braces}) is devoted to   the topological results. In the appendix we discuss a construction of braces  from Fox derivatives.  

This work was  supported by the NSF grant DMS-1664358.

\section{Algebraic preliminaries}\label{Preliminaries}

%We summarize  the algebraic notions used in the paper. 

Throughout the paper we fix a commutative base ring~$R$. By a module we mean an $R$-module and by an algebra we mean an associative $R$-algebra  with unit.  

 \subsection{Quasi-Lie  algebras}\label{Bracketsssy modules}    Given an integer $m\geq 1$, an   \emph{$m$-bracket} in a module~$M$ is   a   map  $ {M}^{m}\to {M}$ which is linear in all $m$ variables.  Here
 $M^m$ is  the  direct product of~$m$ copies of~$M$.  
An $m$-bracket $\mu$ in~$M$     is \emph{cyclically symmetric} if  for all  $x_1,..., x_m\in {M}$, we have
$$\mu(x_1 , \ldots ,  x_{m-1},x_m)=   \mu(x_m , x_1, \ldots , x_{m-1}  ).$$
%A similar  definition applies to algebra brackets. 
A $2$-bracket $\mu$ in~$M$ is \emph{skew-symmetric} if $\mu(x_1 , x_2)=-\mu(x_2, x_1)$
for all  $x_1,  x_2\in {M}$.

   Following \cite{Tu4},    a \emph{quasi-Lie algebra} is    a  module~$M$ carrying a  skew-symmetric 2-bracket $[-,-]$  and   a  cyclically symmetric  3-bracket
$[-,-,-]$ such that
  for any $x,y,z \in M$, we have the following quasi-Jacobi identity:
\begin{equation}\label{Jaco} [[x,y],z] + [[y,z], x]+ [[z,x], y]= [x,y,z]- [z, y,x] .  \end{equation}  
Note that both sides of \eqref{Jaco}  are   cyclically symmetric and  the left-hand side  is the usual  Jacobiator of the 2-bracket $[-,-]$.     Such a pair  of brackets $[-,-], [-,-,-]$    is called  a \emph{quasi-Lie pair}. We recover the usual Lie algebras   when  $[-,-,-]=0$. 

A   \emph{quasi-Lie  algebra homomorphism} from a quasi-Lie  algebra $M$   to a   quasi-Lie  algebra $N$ is a  linear map $f: M \to N$ such that  $[f(x), f(y)]=f([x,y])$ and $[f(x), f(y), f(z)]=f([x,y,z])$ for all $x,y,z\in M$.

 \subsection{Weak derivations} \label{coordrerdd2w}  For an algebra~$A$, we let $A' $ be the submodule of~$A$ spanned by the commutators $ \{xy-yx \, \vert \,  x,y \in A\}$. The quotient module $ \check A = A/ A'$  is the zeroth Hochschild homology of~$A$. 
 A {\it derivation} of~$  {A}$
 is a linear map $d:{   A}\to {   A}$ such that
   $d(xy)= d(x) y+  x d(y)$    for all $x,y \in  {   A}$.    
The identity
$$d(xy-yx )=  d(x) y - yd(x)  +   x d(y) -d(y)x     $$ shows that
the  derivations  of~$ { A}$ carry $A' $ into itself  and   induce  linear endomorphisms of the module  $\check A$. 
A linear   endomorphism of  $ \check A $ is    a \emph{weak  derivation} if it is induced by a derivation  of~$  { A}$.

  \subsection{Braces} \label{coordrerdd2w}   An    \emph{$m$-brace} in an algebra~$A$ is  an $m$-bracket    in the module~$\check A$   which is a weak  derivation  in all~$m$ variables. Thus, an $m$-brace in~$A$ is a map $\mu:(\check A)^m \to \check A$ such that     for any $1\leq j\leq  m$ and   $x_1,..., x_{j-1}, x_{j+1},...,x_{m} \in  \check A$, the     map $$\check A \to \check A, \, \, x\mapsto \mu (x_1, ... , x_{j-1} , x , x_{j+1} , ... ,x_m)$$ is a weak derivation.  
 A $1$-brace in~$A$ is just  a weak derivation   $\check A \to \check A$.

%For commutative algebras, the notion  of a brace considerably simplifies.   

If~$A$ is a  commutative algebra, then $   \check A=A$ and  an  $m$-brace in~$A$  is  an $m$-bracket $A^m\to A$   which is a    derivation  in all~$m$  variables. The following    lemma fully describes   braces in    polynomial algebras.

\begin{lemma}\label{veryfirst}   Let~$X$ be a set and  let $A=R[X]$ be the (commutative) algebra of polynomials in the variables $\{x\}_{x\in X}$ with coefficients in~$R$. For any $m\geq 1$ and any mapping $ f:X^m \to A$, there is a unique $m$-brace $F:A^m\to A$ in~$A$ such that $F(x_1,..., x_m)=f(x_1,..., x_m)$ for all $x_1,..., x_m\in X$. 
\end{lemma}

  \begin{proof} The uniqueness of~$F$  is clear  as the set~$X$ generates the algebra~$A$. We define~$F$ by the following explicit formula: for any   $a_1, ..., a_m \in A$, set 
  $$F(a_1,..., a_m) =\sum_{x_1,..., x_m \in X} \big( \prod_{i=1}^m  \frac{\partial a_i}{\partial x_i} \big ) \,  f(x_1,..., x_m).$$
Here the right-hand side has only a finite number of non-zero terms as each $a_i$
 is a polynomial  in the variables $\{x\}_{x\in X}$ and therefore ${\partial a_i}/{\partial x} \neq 0$ for only a finite set of $x\in X$.
That~$F$  is a    derivation  in all  its variables   follows  from the Leibnitz formula for the partial derivatives.
\end{proof}

%A \emph{bracket homomorphism} from a bracket module~$M$ to a bracket module~$N$ is a  linear map  $f:M  \to N$ 

  \subsection{Quasi-Poisson algebras}\label{Teadsrmy}  
A   \emph{quasi-Poisson algebra} is an algebra~$A$  endowed with a quasi-Lie pair of brackets in  the module ~$\check A$ which both are braces in~$A$.  Such a pair  of braces is called  a \emph{quasi-Poisson pair}.
%Brace algebras with zero 3-bracket   were introduced   by  Crawley-Boevey \cite{Cb} who calls them   $H_0$-Poisson algebras.
A commutative quasi-Poisson algebra  with zero 3-brace is a  Poisson algebra in the usual sense.

  A   \emph{quasi-Poisson  algebra homomorphism} from a quasi-Poisson  algebra $A$   to a   quasi-Poisson  algebra $B$ is a   quasi-Lie algebra homomorphism $ \check A \to \check B$.

 The definition of a quasi-Poisson algebra  given above differs   from parallel  definitions  in \cite{AKsM}, \cite{MTnew}. The main point of difference is that here we not involve actions of Lie groups or Lie algebras. Our quasi-Poisson algebras   generalize   $H_0$-Poisson   algebras introduced by  W. Crawley-Boevey \cite{Cb}. In our  terminology, an $H_0$-Poisson structure  in an algebra~$A$  is a   Lie   bracket  in~$ \check  A$ which is a brace in~$A$.

\section{Trace algebras}\label{coordinate2}\label{AMdddfT0}

We    recall  representation schemes  and trace   algebras following \cite{VdB}, 
   \cite{Cb}. Then we   discuss braces in trace  algebras.

 \subsection{Representation schemes}\label{coordinateddd2w-}  
An   algebra~$A$  and an integer $n\geq 1$ determine an affine scheme  ${\rm {Rep}}_n (A)$, the  \emph{$n$-th  representation scheme} of~$A$.   For each commutative algebra~$S$,  the set of $S$-valued points of ${\rm {Rep}}_n (A)$  is the set of
 algebra homomorphisms  $A \to {\rm Mat}_n(S)$. The coordinate algebra, $A_n$,  of the affine scheme ${\rm {Rep}}_n (A)$
is    generated over~$R$  by
the symbols $x_{ij}$ with $x\in A$ and $i,j  \in  \{1,2, \ldots, n\}$. These generators commute  
and satisfy the following relations:     $1_{ij}=\delta_{ij}$ for all $i,j$,  where
$\delta_{ij}$ is the Kronecker delta; for all $x,y\in A$, $ r\in R $,  and $i,j\in \{1,2, \ldots, n\}$,
$$  (rx)_{ij}= r x_{ij}, \quad    (x + y)_{ij}=x_{ij}+ y_{ij} \quad {\rm {and}} \quad (xy)_{ij}= \sum_{l=1}^n \,  x_{il}\,
y_{lj}. $$
The function on the set of $S$-valued points of ${\rm {Rep}}_n (A)$  determined by    $x_{ij}$ assigns to a homomorphism $f: A \to {\rm Mat}_n(S) $   the $(i,j)$-entry of the  matrix $f(x)  $.   That these functions satisfy the relations   above  is straightforward. 

The  action of the group $G=GL_n(R)$  on 
$ \Hom (A,  {\rm Mat}_n(S))$ by conjugations  induces an action of~$G$ on    $A_n$ for all~$n$. Explicitly, for $g=(g_{kl})\in G$ and any $x\in A, i,j \in \{1,..., n\}$ we have
$$g \cdot x_{ij}=\sum_{k, l=1}^n  g_{ik} (g^{-1})_{lj} x_{kl}.$$
The set of invariant elements $A_n^G=\{a\in A_n\, \vert \, Ga =a \}   $ is a subalgebra of $A_n$. 
This   is     the coordinate algebra of the affine quotient scheme ${\rm Rep}_n (A)//G$  which we view as the  moduli  space of $n$-dimensional representations of~$A$.

The linear map $A\to A_n, x \mapsto
\sum_{i=1}^n  x_{ii} $   is      called the \emph{trace} and denoted    ${\rm tr}$.  The trace  
 annihilates all the commutators in~$A$ and so  ${\rm tr}(A')=0$. 
Thus, the trace induces a
  linear map $ \check A \to A_n$     also denoted~${\rm tr}$. The subalgebra of $A_n$ generated by ${\rm {tr}} (A) ={\rm {tr}} (\check A)  $ is  denoted $A^t_n$ and is called the \emph{$n$-th trace algebra of~$A$}.  A direct computation shows that  ${\rm {tr}} (A) \subset A_n^G $ and therefore $A^t_n \subset A_n^G $.   If  the ground ring~$R$ is an algebraically closed field of  characteristic zero and~$A$ is a finitely generated algebra,  then  a theorem of   Le Bruyn and Procesi \cite{Pro}   implies that    $A^t_n=A^G_n$ so that  $A^t_n$ is         the coordinate algebra of     ${\rm Rep}_n (A)//G$.

  \subsection{Braces in   trace algebras}   The following lemma - inspired by W. Crawley-Boevey \cite{Cb} -    will allow us  to construct braces in the trace algebras.
  
\begin{lemma}\label{Cbtlemmmmheor}  For any integers $m, n\geq 1$ and any   $m$-brace $\mu$ in an algebra~$A$, there is a unique   $m$-brace $\mu_n$ in the algebra  $ A_n^t$ 
such that the following diagram commutes: 
\begin{equation}\label{dia1}
\xymatrix{
{\check A}^m \ar[rr]^-{\mu} \ar[d]_{{ \rm{tr}}  \times \cdots \times  {\rm{tr}}} & & \check A\ar[d]^-{{ {\rm{tr}}}}\\
{(A^t_n)^m} \ar[rr]^-{\mu_n}&  &A^t_n.
} \end{equation}
If $\mu$ is cyclically symmetric, then so is $\mu_n$. If $m=2$ and~$\mu$ is skew-symmetric, then $\mu_n$ is skew-symmetric.
\end{lemma}

  \begin{proof} The uniqueness of $\mu_n$ is clear  as the set ${\rm {tr}} (A)$ generates the algebra $A^t_n$.    We first prove the existence of $\mu_n$  for $m=1$.   We need to show  that  for any weak derivation $\mu:\check A\to \check A$, there is a   derivation of $  A^t_n$ carrying ${\rm {tr}} (x)$ to $ {\rm {tr}} (\mu(x))$ for all $x\in \check A$.   Pick a  derivation $d:A\to A$ inducing $\mu$.    Then  there is a unique  derivation $\widetilde d:A_n\to A_n$ such that $\widetilde d (a_{ij})= (d(a))_{ij}$ for all $a\in A, i,j\in \{1,..., n\}$. Indeed, this formula defines~$\widetilde d$ on the generators of the algebra~$A_n$;   the compatibility with the defining  relations    is straightforward, see \cite[Lemma  4.4]{Cb}.  Clearly,   $\widetilde  d ({\rm {tr}} (a)) ={\rm {tr}} (d(a))$ for all $a\in   A$.  Therefore $\widetilde d(A^t_n) \subset A^t_n$ and the restriction of $\widetilde d$   to $A^t_n$ is a derivation of the algebra $A^t_n$  satisfying our requirements.

Suppose now that $m\geq 2$ and consider the (commutative) polynomial algebra $B=R[\check A]$ in the
variables $\{x\}_{x\in \check A}$. By
 Lemma~\ref{veryfirst},     there is a unique $m$-brace $F:B^m\to B$  such that $$F(x_1,..., x_m)=\mu (x_1,..., x_m) \in \check A\subset B  \quad {\text {\rm  for all}} \quad  x_1,..., x_m\in \check A.$$
 Let $\tau:B\to A^t_n$ be the algebra homomorphism carrying each  $x\in \check A\subset B$ 
 to ${\rm {tr}}(x) \in A^t_n$.   Claim: if $y_1\in \Ker \tau$, then    $  \tau F(y_1,y_2, ..., y_m) =0$ for all $y_2,..., y_m\in B$.  Since~$F$ is a derivation in $y_2,..., y_m$, it suffices to prove this  claim
  when  $y_2,..., y_m \in  \check A\subset B$. Consider   the linear endomorphism $x\mapsto  \mu(x, y_2, ..., y_m)$ of~$\check A$. Since~$\mu$ is a brace in~$A$,  this endomorphism is a weak derivation. 
 By the case $m=1$ discussed above, there is a   derivation $D: A^t_n \to A^t_n$ such that   $$ D({\rm {tr}} ( x)) ={\rm {tr}} ( \mu(x, y_2, ..., y_m))   $$  for all  $ x \in \check A$.
 Then for 
  any finite sequence $x_1, ..., x_N \in \check A$, we have
 $$\tau  F(\prod_{l=1}^N  x_l, y_2,..., y_m) = \tau  \big ( \sum_{k=1}^N   F(x_k, y_2,..., y_m)  \prod_{l\neq k} x_l \big )$$
 $$= \sum_{k=1}^N   \tau  F(x_k, y_2,..., y_m) \prod_{l\neq k} {\rm {tr}}(x_l)= \sum_{k=1}^N   {\rm {tr}}  (\mu(x_k, y_2,..., y_m)) \prod_{l\neq k} {\rm {tr}}(x_l)$$
 $$
 =\sum_{k=1}^N   D({\rm {tr}} ( x_k)) \prod_{l\neq k} {\rm {tr}}(x_l) 
 =D(\prod_{l=1}^N {\rm {tr}}(x_l) ) =D\tau (\prod_{l=1}^N  x_l).$$
Since each $y \in B$ is a linear combination of monomials in the generators $\{x\}_{x\in \check A}$ we have  $\tau  F(y, y_2,..., y_m) = D\tau (y)$. For   $y_1\in \Ker \tau  $, we get 
 $\tau F(y_1, y_2,..., y_m)  =0$. 
 Similar arguments show that if $y_i\in \Ker \tau$ for some $i=1,..., m$ then $ \tau F(y_1,..., y_m) =0$. This  and the surjectivity of~$\tau$ imply  the existence of a 
 mapping $\mu_n: (A^t_n)^m \to A^t_n$ such that the following diagram commutes: 
$$
\xymatrix{
B^m \ar[rr]^-{F} \ar[d]_{\tau \times \cdots \times \tau} & & B\ar[d]^-{\tau}\\
{(A^t_n)^m} \ar[rr]^-{\mu_n}&  &A^t_n.
} $$
  Since~$F$ is a derivation in all variables and $\tau$ is an algebra epimorphism,   $\mu_n$   is a derivation in  all variables, i.e.,  a brace. Restricting the   latter diagram   to $\check A  \subset B $ we obtain the diagram  \eqref{dia1}.   
 The last two claims of the lemma are straightforward.
   \end{proof}
  
\begin{lemma}\label{Cbtlemyahhhmmmheor} Let $Z$ be a commutative algebra  carrying a  skew-symmetric   2-brace
$[-,-]$  and a cyclically symmetric   3-brace  $[-, -,-]$.   If  the  relation   \eqref{Jaco} holds for   elements of a generating set of~$Z$, then it holds for all elements of~$Z$.    
\end{lemma}

  \begin{proof}  Let $L(x,y,z)$ and  $R(x,y,z)$ be respectively the left  and the right  hand-sides of     \eqref{Jaco}.   Since    both sides  are linear in $x,y,z$ and cyclically symmetric, it suffices to verify the following: if  \eqref{Jaco} holds for  triples $x,y,z\in Z$ and $x,y,t\in Z$, then it holds for the triple $x, y, zt$.
  Since $[-,-]$ is a  brace in a commutative algebra, 
  $$ [[x,y],zt]  =  z  [[x,y],t] +  [[x,y],z] t, $$
   $$[[y,zt], x] =    [z [y,t],x] +[[y,z]t, x] $$ $$= z[  [y,t],x] +[z ,x][y,t]+ [y,z]   [t, x]  +[[y,z],x] t. $$
Similarly, 
 $$ [[zt,x], y] =  [z[t,x], y] + [[z,x]t,y]$$
 $$= z[[t,x],y]+ [z,y][t,x] +[z,x][t,y]+ [ [z,x], y]t.$$
  Adding these three  expansions    and using the skew-symmetry of $[-,-]$, we get
    $$L(x,y,zt)= z L(x,y,t)+L(x,y,z)t. $$
Thus,~$L$  satisfies the Leibnitz rule  in the last variable.
 Since the   bracket $[-,-,-]$ also satisfies this rule, so does $R(x,y,z)=[x,y,z]-[z,y,x]$.  Consequently,   if  \eqref{Jaco} holds for the triples $x,y,z $ and $x,y,t $, then it holds for the triple   $x, y, zt$.
   \end{proof}

\begin{theor}\label{Cbtheor}  For any quasi-Poisson algebra~$A$ and    integer  $n\geq 1$, there is a unique  structure of a quasi-Poisson algebra in  $A_n^t$ such that   the trace  ${\rm {tr}}:  \check A  \to A^t_n$  is a quasi-Poisson algebra homomorphism. 
\end{theor}

\begin{proof} 
     Let $[-,-],  [-,-,-]$ be the  given   braces  in~$\check A$ forming a quasi-Lie pair.  By Lemma~\ref{Cbtlemmmmheor}, there are unique       braces  $[-,-],  [-,-,-]$    in   $ A_n^t$  such that  $$[{\rm {tr}} (x), {\rm {tr}}(y)]={\rm {tr}} ( [x,y])\quad  {\rm {and}} \quad [{\rm {tr}} (x), {\rm {tr}}(y), {\rm {tr}} (z)]={\rm {tr}} ([x,y, z])$$ for all $x,y, z\in \check A$.   
Thus,   the quasi-Jacobi relation   \eqref{Jaco} holds  for all  elements of the  set $ {\rm {tr}} (\check A) \subset  A^t_n $ generating $ A^t_n$.   Lemma~\ref{Cbtlemyahhhmmmheor} implies that \eqref{Jaco} holds for all elements of $  A^t_n$.  Since  the brace $[-,-]$  in~$\check A$ is skew-symmetric and the brace $ [-,-,-]$  in~$\check A$ is  cyclically symmetric, so are the induced braces     in $A_n^t$. Therefore  these braces form a quasi-Poisson pair.           They turn   $A^t_n$  into a quasi-Poisson algebra satisfying the conditions  of the theorem. \end{proof}

When the quasi-Poisson algebra~$A$ in Theorem~\ref{Cbtheor} has  zero  3-bracket, the induced 3-bracket in
$A_n^t$  also is  zero. Then the   2-bracket  in $\check A$ is a Lie bracket and so is   the induced 2-bracket in $ A_n^t$. So,    $A_n^t$ is a Poisson algebra in the usual sense. This case of Theorem~\ref{Cbtheor} is   due to W. Crawley-Boevey \cite{Cb}.

% He uses the term  $H_0$-Poisson structures for Lie brackets in   $\check A$ which are braces in~$A$. 

 \subsection{Remark}  A smooth vector field~$v$ on  a smooth manifold~$N$ induces a derivation $d_v$  of the commutative algebra $C^\infty(N)$ of smooth $\RR$-valued functions on~$N$. By definition,  $d_v(f)=df(v)$ for  $f\in  C^\infty(N)$. The map $v \mapsto d_v$ defines  a Lie algebra isomorphism from the Lie algebra of smooth vector fields on~$N$ (with the Jacobi-Lie bracket) onto  the Lie algebra of derivations of $ C^\infty(N) $ with the Lie bracket $[d_1,d_2]=d_1 d_2-d_2 d_1$. For any  algebra~$A$ and any $n\geq 1$, this suggests to view    derivations of the trace  algebra $A^t_n$ as vector fields on (the smooth part) of the affine   scheme  ${\rm Rep}_n (A)//G$.  More generally,   we can view $m$-braces in $  A^t_n$ as  $m$-tensor fields on ${\rm Rep}_n (A)//G$ for all $ {m\geq 1}$.

\section{The main theorem}\label{A topological example2}
  
  We define braces in the modules of loops, recall quasi-surfaces from \cite{Tu3}, and state our main results.
  
  \subsection{Braces in the module  of loops}\label{Case}\label{prel1}
  A \emph{loop} in a  topological space~$X$  is a continuous map $a:S^1\to X$ where the circle $S^1=\{p\in \CC\, \vert \, \vert p \vert=1\}$ is  oriented counterclockwise. 
Two loops $a,b:S^1 \to X$ are \emph{freely homotopic} if there is a continuous map $F:S^1\times [0,1]\to X$  such that $F(p,0)=a(p)$ and $F(p,1)=b(p)$ for all $p\in S^1$.  %Such a map~$F$ is   a \emph{homotopy} or a \emph{deformation} of~$a$ into~$b$.   
 The set of free homotopy classes of loops  in~$X$  is denoted by   $L(X)$. The free module with basis $L(X)$  is denoted by  $M(X)$.

For path-connected~$X$,   we define braces in   $M(X)$ as follows.  Pick a  point $\ast \in X$ and set $\pi=\pi_1(X, \ast)$. For the group algebra $A=R[\pi]$,  the module $A' \subset A $ is generated by the set $\{ uv-vu\, \vert \, u,v \in \pi\}$.  Since   $uv=u(vu)u^{-1}$ for   $u,v\in \pi$, the module $A' $ is generated by  the set $\{uw u^{-1}-w \, \vert \, u,w \in \pi\}$.  Thus,   $\check A =A/A'=R\check \pi$ is the  module  freely generated by   the set~$\check \pi$ of   conjugacy classes of elements of~$\pi$. Note that the   map $  \pi \to L(X)$ carrying  the homotopy classes of  loops to their free homotopy classes    induces  a  bijection $\check \pi = L(X)$. Thus  $\check A=R\check \pi=M(X)$.  A  bracket  in the module  $M(X) $  is   a   \emph{brace} if it is a brace in the algebra~$A$. It is straightforward to see that this definition does not depend on the choice of the base  point~$\ast$.
 
 %This will  allow  us to construct braces in the algebra~$A$     in terms of     loops in~$X$. %We will pursue this approach in the next section. 

\subsection{Quasi-surfaces}\label{QSQS}    By a   \emph{segment} we   mean  a closed segment   and by a   \emph{surface} we   mean  a  smooth oriented  2-manifold  with   boundary.      A \emph{quasi-surface}  is a path-connected  topological space  obtained by gluing a surface~$\Sigma$ to  a  topological space~$Y$  along a continuous mapping from a union of several    disjoint     segments  in $\partial \Sigma$  to~$Y$.  We call~$\Sigma$ the \emph{surface core}  and~$Y$ the \emph{singular part} of~$X$. 
  For example, taking a path-connected surface~$\Gamma$   and collapsing several   disjoint   subsegments  of  $\partial \Gamma$ into a single point we obtain a quasi-surface with 1-point singular part. When only one segment in $\partial \Gamma$ is collapsed, the resulting quasi-surface is homeomorphic to~$\Gamma$. Another way to turn~$\Gamma$ into a quasi-surface is as follows:
pick a  surface $\Sigma \subset \Gamma$ meeting $ \overline {\Gamma  \setminus \Sigma}  $  at a   finite non-empty set of disjoint segments with endpoints in $\partial \Gamma$. Taking~$\Sigma$ as the   surface core and $\overline {\Gamma  \setminus \Sigma} $ as the singular part,   we turn~$\Gamma$ into a  quasi-surface.   In particular, given  a finite tree $T\subset \Gamma$ meeting $\partial \Gamma$ at the vertices of degree~$1$, we can take  a closed regular neighborhood of~$T$ in~$\Gamma$ as the surface core and take the closure of the rest of~$\Gamma$ as the singular part. In this way, $T$ gives rise to  a structure of a quasi-surface on~$\Gamma$. 

\subsection{Main results}\label{QSmainQS}    For any  quasi-surface~$X$, the author constructed in~\cite{Tu3}  a skew-symmetric 2-bracket $[-,-,]_X$ and cyclically symmetric $m$-brackets $\{\mu^m\}_{m\geq 1}$  in  the module $ M(X)$.  
We  now state our main theorem.

\begin{theor}\label{MAIN+}   The brackets $[-,-,]_X$ and  $\{\mu^m\}_{m\geq 1}$ in $ M(X)$ are braces.  
\end{theor}

   By Theorem~4.2 of \cite{Tu4},  the brackets $[-,-,]_X$ and   $  \mu^3$  form a quasi-Lie pair. Combining  with Theorem~\ref{MAIN+}  we obtain the following claim.

\begin{corol}\label{Main+0-}    The  group algebra  $ R[\pi_1(X)]$ endowed with the braces  $[-,-,]_X$ and $ \mu^3$     is a quasi-Poisson algebra.
 \end{corol}
 
Corollary~\ref{Main+0-}  and  Theorem~\ref{Cbtheor}  imply  the following claim.
 
 \begin{corol}\label{Main+0-+}    For the algebra $A=R[\pi_1(X)]$   and any integer $n\geq 1$,  there is a unique structure of a  quasi-Poisson algebra
    in   $A_n^t$ such that   the trace  ${\rm {tr}}:  \check A  \to A^t_n$  is a quasi-Poisson algebra homomorphism. 
 \end{corol}

 %For surfaces,  this  construction  is known to yield usual   Poisson brackets  on the smooth  parts  of the %moduli spaces, see  \cite{MTnew}.

In the next two  sections we  recall    the definitions of the brackets $\{\mu^m\}_{m\geq 1}, [-,-,]_X$  and   prove Theorem~\ref{MAIN+}.

 \section{Gates}\label{A topological example}

We  define  gates in  an arbitrary  topological space~$X$ and show how gates give rise to braces.

 \subsection{Gates}\label{topolsett}\label{coordinateddd2w-}   
    A \emph{cylinder neighborhood} of a   set $C \subset X$  is   a  pair consisting of a closed set $U\subset X$  with $C \subset  {\rm {Int}} (U)$ and     a homeomorphism  $ U\approx C\times [-1,1]$    carrying  ${\rm {Int}} (U)$ onto  $C\times (-1, 1)$ and carrying each point $c\in C   $ to $(c,0)  $.    A  \emph{gate} in~$X$ is  a closed path-connected  subspace~$C$   of~$X$ endowed with a cylinder neighborhood   and such that all loops in~$C$ are contractible in~$X$.  We  will identify the   cylinder neighborhood in question with $C\times [-1,1]$ via the given homeomorphism.     
    
 For a gate $C \subset X$, consider     the  map $H:X \to S^1 $  carrying  the complement of  $ C\times (-1,1) $  in~$X$ to $-1\in S^1$ and  carrying $C\times \{t\}  $ to ${\rm{exp} }(\pi i t)\in S^1$ for all $t\in [-1,1]$. 
We say that a  loop $a:S^1 \to X$ is \emph{transversal} to~$C $  if   the map $Ha: S^1 \to S^1$   is transversal to $1  \in S^1$. Then   the set $a^{-1}(C) =(Ha)^{-1}(1)$  is finite.   For each  $p\in a^{-1}(C)$,  we   define the {\emph {crossing sign}} $\varepsilon_p(a) $: if  at~$p$ the loop~$a$    goes  from $C\times (-1,0) $ to $C \times (0,1)$  then   $\varepsilon_p(a)=+1$, otherwise,  $\varepsilon_p(a)=-1$.

%The number
%\begin{equation}\label{apm-} a\cdot C= \sum_{p\in a^{-1}(C)} \varepsilon_p(a)   \in {\mathbb Z}  \end{equation}
%is   the   algebraic  intersection number of~$a$ and~$C$.   
%Note  that the formula $a\mapsto a \cdot C$ defines a homomorphism $H_1(X)\to \mathbb {Z}$.      An example of a gate is provided by a   simply connected proper codimension~1    submanifold~$C$ of a manifold together with a suitable homeomorphism of a closed neighborhood of~$C$   onto $C\times [-1,1]$. In this case, $a\cdot C$ is the usual intersection number of a  loop~$a$ with~$C$.  

  \subsection{Gate brackets} \label{Gate brackets-}  We start with  notation.   For any loop $a:S^1\to X$ we call the point $a(1)\in X$   the \emph{base point} of~$a$.  For   $p\in   S^1$,  we let $a_p: S^1\to X$ be the  loop  obtained as the composition of~$a$ with the rotation $S^1\to S^1$ carrying $1\in S^1$ to~$p$.  This loop is based at $a(p)$ and is called a \emph{reparametrization} of~$a$.  
  Set  $   L={L} (X)$    and $ M= M(X) $, see Section~\ref{prel1} for notation. 
 For  any   loop~$a$ in~$X$, we let $\langle a \rangle \in  {L }  \subset  M  $ be the free homotopy class of~$a$.
  
Each    gate $C \subset  X$ determines  
 brackets $\{\mu^m_C\}_{m\geq 1}$ in the module~$M$ as follows.  Pick a  point $\star \in C$ and,  for each   $c\in C$, pick  a path $\gamma_c$   in~$C$ from~$\star$ to~$c$. Given a loop~$a$ in~$X$ with   $a(1) \in  C$ we let $  a^\gamma=\gamma_{a(1)} \,a\, \gamma_{a(1)}^{-1}$ be the loop   based at~$\star$ and obtained from~$a$ by conjugation along  the path $\gamma_{a(1)}$. Since~$C$ is a gate, the homotopy class of $  a^\gamma$     in $\pi_1(X, \star)$   does not depend  on the choice of   $\gamma_{a(1)}$.
 
 Consider  $m \geq 1$   loops $a_1,..., a_m$ in~$X$   transversal to~$C$.   For   any $i=1,...,m$ and   $p_i\in a_i^{-1}(C) \subset S^1$,  we have  the reparametrization $ (a_i)_{p_i}$ of $a_i$  based at $a_i(p_i)\in C$  and   the loop $(a_i)_{p_i}^\gamma=((a_i)_{p_i})^\gamma$ based  at~$\star$.    Set
  %\begin{equation}\label{twotwo3aa}
  $$ \mu^m_C (  a_1  , \ldots ,   a_m ) =
 \sum_{p_1\in a_1^{-1}(C), \ldots, p_m\in a_m^{-1}(C)} \, \prod_{i=1}^m \varepsilon_{p_i}(a_i)\,  \langle \prod_{i=1}^m  (a_i)_{p_i}^\gamma \rangle \in M .  $$   The following claim is proved in \cite{Tu4}, Section~3.
 
 \begin{lemma}\label{strlelemucte}  $\mu^m_C (  a_1  , \ldots ,   a_m )  \in M$ depends only on the free homotopy classes of  the loops $a_1,..., a_m$ and, in particular, does not depend on the choice of the point~$\star$.
 \end{lemma}

Using the product structure in   the cylinder neighborhood   of~$C$,     we easily observe that  each loop in~$X$ is freely homotopic to  a loop   transversal to~$C$. Therefore Lemma~\ref{strlelemucte}  yields a map $$L^m \to M, \,\, (\langle a_1 \rangle,..., \langle a_m \rangle)\mapsto  \mu^m_C (  a_1  , \ldots ,   a_m ) .$$ This map extends by linearity to an  $m$-bracket $\mu^m_C$  in~$  M$. Since $\langle uv \rangle =\langle vu \rangle$ for any loops $u,v $ in~$X$ based at~$\star$,  the bracket $\mu^m_C$  is cyclically symmetric.

  \begin{lemma}\label{strlelemuctess}  The bracket $\mu^m_C   $ in $M $ is a brace.  \end{lemma} 
 
 \begin{proof} Pick a point $\ast  \in X \setminus C$ and  consider the group algebra   
   $A=R[\pi_1(X, \ast)]$.  We need to prove that the bracket $\mu^m_C   $ in $M=\check A$ is a brace in~$A$. Since  $\mu^m_C   $ is  cyclically symmetric, it suffices to show that it is a weak derivation in the first variable.  To this end,  we refine the definition of $\mu^m_C(a_1,..., a_m)$  whenever the loop $a_1$ is   based at~$\ast$.     For any points $p_1\in a_1^{-1}(C), \ldots, p_m\in a_m^{-1}(C)$   consider the  product path 
\begin{equation}\label{tvi} a^{-}_{p_1 } \, \gamma_{a_1(p_1)}^{-1} \, \big(\prod_{i=2}^m  (a_i)_{p_i}^\gamma\big ) \, \gamma_{a_1(p_1) } \, a^+_{p_1 } \end{equation}  where  $a^-_{p_1 }$ is the path in~$X$ obtained as the restriction of~$a_1$ to the arc in $S^1$ leading from $1\in S^1$ to $p_1$ and  $a^+_{p_1 }$ is the path in~$X$ obtained as the restriction of~$a_1$ to the  arc in $S^1$ leading from $p_1$ to $1$. The path \eqref{tvi}
is a loop based at~$\ast$.  Let $ [p_1,..., p_m] \in \pi_1(X, \ast)$  be the homotopy class of this loop.  Set
 \begin{equation}\label{34d} \widetilde \mu^m  (  a_1  , \ldots ,   a_m ) = \sum_{p_1\in a_1^{-1}(C), \ldots, p_m\in a_m^{-1}(C)} \, \prod_{i=1}^m \varepsilon_{p_i}(a_i)\,  [p_1,..., p_m] \in A . \end{equation}  
 Note that any  two $\ast$-based  loops in~$ X$ transversal to~$C$  and homotopic in the class of $\ast$-based loops  may be related by a finite sequence of  homotopies  of the following two  types  (and   inverse homotopies):  (i)  deformations   in the class of $\ast$-based loops transversal to~$C$ and (ii) deformations  pushing a branch of the loop     in $X \setminus C$   across~$C$ and creating   two new crossings   with~$C$. It is clear that homotopies of   $a_1 $ of   type (i) preserve $\widetilde \mu^m (  a_1  , \ldots ,   a_m )$. A homotopy of $a_1$ of  type  (ii)     creates two   new points $p, p'$ in $a^{-1}(C)$ such that  $\varepsilon_{p}(a_1)=-\varepsilon_{p'}(a_1)$ and  $[p, p_2,..., p_m]=[p', p_2,..., p_m]$ for any    $\{p_i\in a_i^{-1}(C)\}^{m}_{i=2}$. Therefore the expression   $ \widetilde \mu^m  (  a_1  , \ldots ,   a_m )$ is preserved under deformations of  $a_1$ in the class of $\ast$-based loops.   So, the formula 
 $a_1 \mapsto  \widetilde \mu^m  (  a_1  , \ldots ,   a_m )$ defines a  mapping $\pi_1(X, \ast) \to A$. This mapping  extends by linearity to a linear map $A\to A$ which is easily seen  to be   a derivation and a lift of the linear map $\check A \to \check A, \langle a_1 \rangle \mapsto \mu^m_C(a_1,..., a_m)$. 
 Thus,  $\mu^m_C   $ is a   weak derivation in the first variable. This proves the lemma. 
 \end{proof} 
 
 \subsection{Remark} The  arguments  in the proof of Lemma~\ref{strlelemuctess}  may be used to show that the sum \eqref{34d} is preserved under free homotopies of the loops $a_2,..., a_m$. Therefore~\eqref{34d} defines a  map $A \times M^{m-1}\to A$   which is a derivation in the first variable and is linear in the other $m-1$ variables.

\section{Quasi-surfaces and proof of Theorem~\ref{MAIN+}}\label{Quasi-surfaces and braces}

%We recall the necessary definitions from \cite{Tu3}, \cite{Tu4} and then prove  .

\subsection{More on quasi-surfaces}\label{Definitions and notation}     
Consider a quasi-surface~$X$ with surface core~$\Sigma$ and singular part~$Y$. We suppose that~$X$ is obtained by gluing~$\Sigma$ to~$Y$  along a continuous map   $   \alpha  \to Y$ where $ \alpha \subset \partial \Sigma $ is a    union of a finite number ($\geq 1$) of    disjoint  closed segments  in $\partial \Sigma$. 
Note that     $Y\subset X$ and   $X\setminus Y =\Sigma\setminus \alpha$.    
We   fix a closed    neighborhood  of~$\alpha$ in~$\Sigma$ and 
  identify   it  with  $\alpha \times [-1,1]$ so that 
  $$\alpha=\alpha \times \{-1\} \quad  \quad {\text {and}} \quad \quad  \partial \Sigma  \cap (\alpha \times [-1,1])=\alpha   \cup (\partial \alpha \times [-1,1]).$$
  The  surface
$$\Sigma'=\Sigma \setminus  (\alpha \times [-1,0)) \subset \Sigma \setminus \alpha \subset X   $$  is  a  copy of $\Sigma$    embedded in~$X$ and disjoint from~$Y$. 
We    provide~$\Sigma'$  with the   orientation induced  from that of~$\Sigma$. 
%We      call~$\Sigma'$  the \emph{reduced surface core} of~$X$. 
 % For examples of quasi-surfaces, see \cite{Tu3}.  

Set  $\pi_0=\pi_0(\alpha)$.   For    $k\in  \pi_0 $,   we   let   $ \alpha_k $ be  the corresponding   segment component of $\alpha \times \{0\} \subset \partial \Sigma' \subset X$. 
 We call the segments  $\{\alpha_k\}_{k\in \pi_0}$ the \emph{gates} of~$X$. It is clear that these segments are gates of~$X$ in the sense of Section~\ref{coordinateddd2w-}.
They       separate  $
\Sigma'\subset X$ from the rest of~$X$.   By Section~\ref{Gate brackets-}, every  gate  $\alpha_k$ determines a sequence of brackets $\{\mu^m_{\alpha_k}\}_{m\geq 1}$ in the module $M=M(X)$. For each $m\geq 1$, we define the \emph{total gate $m$-bracket} of~$X$ to be the sum
$$\mu^m=\sum_{k\in \pi_0} \mu^m_{\alpha_k} :M^m \to M.$$

 We keep the objects  $X, \Sigma, \Sigma', \alpha, \pi_0$ till the end of Section~\ref{Quasi-surfaces and braces}. 
 
\subsection{Loops in~$X$}\label{Generic loops}   For any loop $a: S^1 \to X$ and a gate $\alpha_k$,   we set $a\cap \alpha_k=a(S^1) \cap \alpha_k $.
We  say that an (ordered)  pair of loops $a,b $ in~$X$ is \emph{admissible} if these loops do not meet at the gates and 
for any gate $\alpha_k$ of~$X$  and any points $p\in a\cap \alpha_k$, $q\in b\cap \alpha_k$,   the pair   (a   vector tangent to $\alpha_k$ and  directed from~$p$ to~$q$, a vector at~$p$ directed inside $\Sigma'$) is positively oriented in~$\Sigma'$.

 We say that a loop $a: S^1 \to X$ is    \emph{generic} if  (i)  all  branches of~$a$   in~$\Sigma'$   are smooth immersions meeting $\partial \Sigma'$ transversely   at a finite set of points which all lie  in the  interior  of the gates, and (ii)  all self-intersections of~$a$ in $\Sigma'$ are double transversal intersections   in $\Int(\Sigma')=\Sigma' \setminus \partial \Sigma'$.     A generic loop~$a$     traverses  any    point of a gate   $  \alpha_k$   at most once so that the restriction of the map $a:S^1\to X$ to $a^{-1}(\alpha_k) \subset S^1$ is a bijection onto the set $ a \cap \alpha_k $. In  this context, we adjust   notation of Section~\ref{A topological example} and   use the letter~$p$ for   elements of the set $a  \cap \alpha_k$ rather than for their preimages under~$a$.
Accordingly, the crossing sign $\varepsilon_p(a)$ at $p\in a \cap \alpha_k$      is  $+1$ if~$a$ goes at~$p$  from $X\setminus \Sigma'$  to $\Int(\Sigma')$   and is   $-1$ otherwise. 

A pair of   loops in~$X$    is    \emph{generic}  if    both loops  are generic and all  their intersections   in~$\Sigma'$ are double transversal intersections    in $\Int(\Sigma')$. In particular,   such loops  do not meet at the gates. 
It is easy to see that    each ordered pair of loops in~$X$  is freely homotopic to  an admissible generic pair of loops.

\subsection{The   bracket $[-,-]_X$}\label{The form bullet} 
 We   recall    the     2-bracket $[-,-]_X$  in  the module $M=M(X)$, see~\cite{Tu3}.   For a loop $a: S^1 \to X$ and a  point $r \in X$    traversed by~$a$ exactly once,
 we let $a_r$ be the   loop  which starts at~$r$ and goes along~$a$ until coming back to~$r$. 
  For   any  loops $a,b$ in~$X$ set    $$a\cap b= a(S^1) \cap b(S^1)\cap \Sigma'   .$$  If  $a,b$ is a generic  pair  of loops  then the set $a\cap b \subset   \Int(\Sigma')$ is finite and  each point $r \in a\cap b$ is traversed by~$a$ and~$b$ exactly once so that we can consider the loops $a_r, b_r$ based at~$r$,  their product $a_rb_r$, and the  free homotopy class $\langle a_rb_r \rangle\in L(X)$.    Set $\varepsilon_r  (a,b) =  1$ if  the tangent vectors of~$a$ and~$b$ at~$r$ form a    positive basis in the tangent space of~$\Sigma'$ at~$r$ and   set $\varepsilon_r   (a,b)=-1$ otherwise. Also, a  generic pair of  loops $a,b$ determines a  finite  set of triples   
$$T(a,b)=\{(k,p,q) \, \vert \, k\in \pi_0, p\in a\cap \alpha_k, q\in b\cap \alpha_k \}.$$
     For any triple  $(k,p,q)\in T(a,b)$, we   can   multiply the loops $a_p, b_q$  based at $p,q$ using an arbitrary  path   connecting  $p,q$ in $\alpha_k$. The product  loop determines a well-defined  element of $L (X)$   denoted $\langle a_p b_q \rangle$.

     \begin{lemma}\label{1aeee++} There is a unique 2-bracket $[-,-]_{X}$ in $M=M(X)$ such that for any admissible generic pair of loops $a,b$ in~$X$, we have 
\begin{equation}\label{chorieee++}
[\langle a \rangle, \langle b \rangle]_{X} =  2\sum_{r\in a\cap b} \varepsilon_r  (a,b) \langle a_r b_r \rangle -\sum_{(k,p,q)\in T (a,b) }\,     \varepsilon_p(a)  \,\varepsilon_q(b) \langle a_pb_q \rangle. \end{equation}
\end{lemma}

 The uniqueness of such a bracket  follows from the last observation in Section~\ref{Generic loops}.
 The existence   follows from \cite{Tu3}, Section~4.4.  
 
The bracket $[-,-]_{X}$    generalizes Goldman's  bracket \cite{Go1}, \cite{Go2}: 
its value    on any pair of  free homotopy  classes  of loops in $\Sigma'\subset X$    is   twice  their  Goldman     bracket. 
  
\subsection{Proof of Theorem~\ref{MAIN+}}\label{yahoo}   That  the total gate brackets $\{\mu^m\}_{m\geq 1}$ in $M= M(X)$ are braces   follows from Lemma~\ref{strlelemuctess} and the obvious fact that a  sum of   braces  is a brace.   We need  to prove that the 2-bracket $[-,-]_{X}$  is a brace. 
 Since this bracket is skew-symmetric, it suffices to prove that for any $y\in L=L(X)$ the linear map $x\mapsto [x,y]_X:M \to M$ is a weak derivation. 
 
 Pick an arbitrary  base point $\ast \in Y$; set  $\pi=\pi_1(X, \ast)$ and $A=R[\pi]$. For a loop~$a$ based at~$\ast$, we   let $[a] \in \pi \subset A$ be the homotopy class of~$a$. We  construct a derivation $d:A\to A$ as follows. 
  First, we represent~$y$ by a generic loop~$b$ in~$X$. 
Since~$b $ meets the gates only  in their interior points, every  element of  the group  $\pi=\pi_1(X, \ast)$ can be represented by a   loop~$a$ based at~$\ast$  such that the pair   $a,b$ is generic and admissible.       For any  point $r\in a\cap b$ we define a loop $a\circ_r b $: it  starts at~$\ast$ and goes along~$a$ till~$r$, then it goes along the  whole  loop~$b_r$   back to~$r$, then it continues along the remaining part of~$a$   to~$\ast$. For any triple  $(k,p,q)\in T(a,b)$, we define a loop $a\circ_{p,q} b $:  it  starts at~$\ast$ and goes along~$a$ till~$p$, then it goes along $\alpha_k$ to~$q$, then it goes along the whole loop~$b_q$   back to~$q$, then it follows along $\alpha_k$ back  to~$p$, then it continues along the remaining part of~$a$   to~$\ast$. Set
\begin{equation}\label{chori--eee++}
d([a]) =  2\sum_{r\in a\cap b} \varepsilon_r  (a,b) [a\circ_r b ]  -\sum_{(k,p,q)\in T (a,b) }\,     \varepsilon_p(a)  \,\varepsilon_q(b) [a\circ_{p,q} b  ] \in A . \end{equation}
We claim that (i) $d([a]) \in A$ depends only on $[a]\in \pi$ and does not depend on the choice of the loop~$a$ in its homotopy class;
(ii) the linear extension $A\to A$ of~$d$ is a derivation; (iii) this derivation is a lift of the linear map  $x\mapsto [x,y]_X:M \to M$. The proof of (i) uses the same arguments as the  proof of \cite{Tu3}, Lemma 3.1. 
The key observation is that for any homotopic  loops $a_0, a_1$ such that the pairs $a_0,b$ and $a_1, b$ are admissible, there is a homotopy $(a_t)_{t \in [0,1]}$ of~$a_0$ into~$a_1$ such that the pair $a_t,b$ is admissible for all~$t $.  The proof of (ii) is straightforward: if $a,a'$ are  loops based at~$\ast$ and such that both  pairs   $a,b$ and $a', b$ are admissible and generic, then  the pair $aa', b$   is admissible. Deforming slightly   $a'$ we can  enurse that this pair is  generic. Then the set $(aa')\cap b$ is a disjoint union of the sets $a\cap b$ and $a'\cap b$. Similarly, the set $T(aa', b)$ 
is a disjoint union of the sets  $T(a,b)$ and $T(a', b)$. Using these facts and computing $d ([aa'])$ via~\eqref{chori--eee++} we get $d ([aa'])= d ([a] ) [a'] + [a] d ([a'])$. Finally, Claim (iii) is obtained by direct comparison of Formulas~\eqref{chorieee++} and~\eqref{chori--eee++}.

\subsection{Remark}   Similar arguments  show that the expression~\eqref{chori--eee++} is preserved under homotopies of the loop~$b$. As a result, we obtain  a well-defined bilinear form $A \times M(X)\to A$.   In the case of surfaces, this form was first constructed by  N. Kawazumi and Y. Kuno \cite{KK1}, \cite{KK2}  who proved that it defines a right action   of the Goldman-Lie algebra  of loops  on  the group algebra of the fundamental group. 

\appendix

\section{Fox derivatives}\label{Derivative brackets in group algebras}

 We  define  Fox derivatives in group algebras and explain how they give rise to  braces and how they arise from    gates in topological spaces.

  \subsection{Fox derivatives}\label{Constructions of brackets}  Let~$\pi$  be a group.  A (left)  \emph{Fox derivative} in the group algebra  $A=R[\pi]$  is a linear map $\partial:A\to A$ such that $\partial (xy)=\partial (x) +x \partial (y)$ for all $x,y \in \pi \subset A$.
For arbitrary $x,y \in A$, we have then $\partial (xy)=\partial (x) \, {\rm {aug}} (y)+x \partial (y)$ where 
  ${\rm {aug}}:A\to R$ is the linear map carrying all elements of~$\pi$ to~$1$.
 For   $x\in \pi$,  we can uniquely expand $\partial (x)= \sum_{a\in \pi}    (x/a)_\partial \, a$ where $(x/a)_\partial\in R$ is non-zero only  for a finite set of~$a$. We define a  linear map $\Delta_{\partial}:A\to A$ by 
   $$ \Delta_{\partial}(x)=  \sum_{a\in \pi}   (x/a)_\partial \,  a^{-1} x a \quad {\text {for all}} \quad x\in \pi.  $$
 
  \begin{lemma}\label{sDDDtrbucte}    
   $\Delta_\partial (A')=0$.
   \end{lemma}
  
\begin{proof}   
  It suffices to prove that $\Delta_\partial (xy  -yx)=0$ for any $x,y\in \pi$.
   We have  $$\partial (xy)=\partial (x)  +x \partial (y)= \sum_{a\in \pi}  \big (  (x/a)_\partial \, a+   (y/a)_\partial \, xa\big ).$$
   Therefore, by the definition of $\Delta_\partial$,  
  $$\Delta_\partial (xy  )= 
  \sum_{a\in \pi} \big ( (x/a)_\partial \,  a^{-1} xy a +    (y/a)_\partial \,  (xa)^{-1} xy (xa) \big )$$
  $$ =\sum_{a\in \pi} \big ( (x/a)_\partial \,  a^{-1} xy a +    (y/a)_\partial \,  a^{-1} y xa \big ).$$
  The  latter  expression is invariant under the permutation $x \leftrightarrow y$. So,  $\Delta_\partial (xy  )=\Delta_\partial (yx  )$ and   $\Delta_\partial (xy  -yx)=0$.
  \end{proof}

  The linear map $\check A=A/A'\to A$ induced by $\Delta_\partial:A\to A$ is denoted by $\check \Delta_\partial$.

 \begin{theor}\label{sDDjkgfjfgjDtrbucte} Let $p:A\to \check A$ be the projection.    For any $m\geq 1$
  Fox derivatives $\partial_1,  ..., \partial_m$ in~$A$,  the  map  $\mu^m:\check A^{  m} \to \check A$ defined by  
\begin{equation}\label{cccdmm} \mu^m(x_1 ,     \ldots ,  x_m)=p \big  ( \check \Delta_{\partial_1} (x_1)    \cdots
  \check \Delta_{\partial_m} (x_m) \big ) \end{equation}
  for $x_1,   ..., x_m  \in  \check A$  is  an $m$-brace  in~$A$.
   \end{theor}

\begin{proof}   % Since the right-hand side of  \eqref{cccdmm} is linear in all~$m$ variables, it  does define   a linear map  $\mu^m:\check A^{ m} \to \check A$.  
 We need to prove that  $\mu^m$  is a weak derivation in all variables, i.e.,  that for any  $i=1,...,m$  and $x_1,..., x_{i-1}, x_{i+1},..., x_m\in \check A$, the  map
\begin{equation}\label{678} \check A \to \check A, \,\, x \mapsto \mu^m(x_1,  \ldots   , x_{i-1} , x   ,  x_{i+1}, \ldots   , x_{m})\end{equation} is induced by a derivation in~$  A$. 
   Set $$G= \check \Delta_{\partial_1} (x_1)   \cdots
  \check \Delta_{\partial_{i-1}} (x_{i-1}) \in A \quad {\rm {and}} \quad H= \check \Delta_{\partial_{i+1}} (x_{i+1})   \cdots
  \check \Delta_{\partial_{m}} (x_{im}) \in A.$$
   For $x  \in \pi$, we expand $\partial_i (x)= \sum_{a\in \pi}  (x/a) a$ with $(x/a) =(x/a) _{\partial_i}$. Then 
   %$$\Delta_{\partial_m} (x)=\sum_{a\in \pi}  (x/a) a^{-1} x a $$ and 
   $$\mu^m(x_1,  \ldots , x_{i-1} , x   ,  x_{i+1}, \ldots  ,  x_{m}) =p(  G   \check \Delta_{\partial_{i}} (x) H) $$
$$ =p(    \sum_{a\in \pi}  (x/a)   G  a^{-1} x a H) =p( \sum_{a\in \pi} (x/a) a HG  a^{-1} x )$$ where we use that $p(   G  a^{-1} x aH) =p(  a  HG  a^{-1} x )$.
   Thus,  the  map \eqref{678}  is induced by the linear map $ A\to A$ carrying any $x\in \pi$ to $    \sum_{a\in \pi} (x/a) a GH  a^{-1} x.  $ This map is a derivation. In fact,   for any $F\in A$, the linear map $ d:A\to A$ carrying any $x\in \pi$ to $    \sum_{a\in \pi} (x/a) a F  a^{-1} x  $ is a derivation in~$A$.
   Indeed, for
$x, y \in \pi$, we have 
$$d (x)=  \sum_{a\in \pi} (x/a) a F  a^{-1} x\,\, \,\, \,  {\rm {and}} \,\,\,    \quad d(y)=\sum_{a\in \pi} (y/a)  a F a^{-1} y.$$ Also, 
$$\partial_i(xy) =\partial_i(x) +x \partial_i (y)= \sum_{a\in \pi}  \big ( (x/a) a +   (y/a) x a \big )$$  
    and so 
   $$ d (xy)=\sum_{a\in \pi}  \big ( (x/a) a F a^{-1} xy  +   (y/a)  xa F (x a)^{-1} xy \big )   =d(x) y +x d(y).$$
  This completes the proof of the theorem. \end{proof}  
  
  %    Thus, $d$ is a derivation in~$A$.

   Combining Theorem~\ref{sDDjkgfjfgjDtrbucte} with  Lemma~\ref{Cbtlemmmmheor}  we obtain the following.

\begin{corol}\label{sDDjmmmkgfjfgjDtrbucte}    For any 
  integers  $m,n\geq 1$  and  any  Fox derivatives $\partial_1, ..., \partial_m$ in~$ A$,  there is a unique 
$m$-brace $\mu^m_n$ in $ A_n^t$ such that  
 $$\mu^m_n({\rm tr} (x_1) ,    \ldots  , {\rm tr} ( x_m))={\rm tr} (    \Delta_{\partial_1} (x_1)      \cdots
   \Delta_{\partial_m} (x_m)   )  $$
   for all  $x_1,  ..., x_m \in A$.
 \end{corol}
   
%If  $\partial_1=  \cdots = \partial_m$, then the $m$-braces $\mu^m$ and $\mu^m_n$ are cyclically symmetric. This follows from the identities $p(xy)=p(yx)$ and 
  % ${\rm tr} ( xy)={\rm tr} ( yx)$ for all $x,y\in A$.

%\subsection{Equivalence of Fox derivatives}\label{topdequivolsett} 

\subsection{Gate derivatives} Consider  a topological space~$X$ with base point~$\ast $. A \emph{based gate} in~$X$ is  a gate $C \subset X \setminus  \{\ast\} $    endowed with  a path $\gamma:[0,1]\to X$ such that $\gamma (0)=\ast$   and $\gamma ( 1) \in C$.  We show   that such a pair $(C, \gamma)$ gives rise to a Fox derivative   in the group  algebra $ R[\pi]$  where $\pi=\pi_1(X, \ast)$. 
  
Pick any   loop $a:S^1\to X$ based at $\ast$ and  transversal to~$C$.  For   $p\in a^{-1}(C) \subset S^1 $,  we  let $a^p$ be  the product of the following three  paths in~$X$:   the    restriction of~$a$ to the arc $(1,p) \subset S^1$; any path~$\beta$  in~$C$ from $a(p) $  to  $\gamma ( 1) $;   the path  inverse to~$\gamma$. Clearly,  $a^p$ a loop  in~$X$  based at~$\ast$. Since all loops in~$C$ are contractible in~$X$, the element $[a^p]$ of~$ \pi$ represented by this loop  does not depend on the choice of~$\beta$. 
  Set 
  \begin{equation}\label{tweeotwo3} \partial(a)= \sum_{p \in  a^{-1}(C)} \varepsilon_{p}(a) \, [a^p] \in  R [\pi]     \end{equation}
  where $\varepsilon_{p}(a)=\pm 1$ is the  crossing sign  defined in   Section~\ref{coordinateddd2w-}.

  \begin{lemma}\label{svvvvtrbucte}  Formula \eqref{tweeotwo3} determines a well-defined map $\partial : \pi \to R [\pi]$. Its linear extension $ R [\pi] \to R [\pi]$   is a Fox derivative.  
\end{lemma}
  
\begin{proof} It is clear that   all elements of~$\pi$  can be represented by loops based at $\ast$ and transversal to~$C$.  The arguments  in the proof  of Lemma~\ref{strlelemuctess} can be used to show   that  if two   loops    $a,a'$ based at $\ast$ and transversal to~$C$  represent the same element of~$\pi$, then 
 $\partial (a)=\partial (a')$. If $a,b$ are   loops  based at $\ast$ and     transversal to~$C$, then   so is their product,  and it follows    from the definitions that 
$\partial(ab)= \partial(a) + a \partial (b)$. This implies the second claim of the lemma. 
 \end{proof} 

Combining Lemma~\ref{svvvvtrbucte} with the results of Section~\ref{Constructions of brackets} we conclude   that each   sequence of $m\geq 1$ based gates in~$X$  (not necessarily disjoint or   distinct)  gives rise to  an  $m$-brace in $A=R[\pi]$  and to  $m$-braces in the    algebras $\{A^t_n\}_n$.   In particular, 
a sequence of~$m $ copies of the same based gate $ (C, \gamma)$ determines a cyclically symmetric  $m$-brace  in~$A$. We  leave it to the reader to verify that this brace does not depend on~$\gamma$ and coincides with the $m$-brace  $\mu^m_C$  defined  in Section~\ref{Gate brackets-}.


\begin{thebibliography}{CJKLS}


 

\bibitem[AKsM]{AKsM}
A. Alekseev, Y. Kosmann-Schwarzbach, E. Meinrenken, \emph{Quasi-Poisson manifolds.}
Canad. J. Math. 54 (2002), no. 1, 3--29.


\bibitem[AB]{AB} M. F. Atiyah, R. Bott,  \emph{The Yang-Mills equations over Riemann surfaces.} Philos.
Trans. R. Soc. Lond. Ser. A 308,  (1983), 523–615.

\bibitem[Cb]{Cb}
W. Crawley-Boevey,
\emph{Poisson structures on moduli spaces of representations.}
J. Algebra 325 (2011), 205--215.



%\bibitem[EH]{EH}
%D. Eisenbud, J. Harris, \emph{The geometry of schemes}. Graduate
%Texts in Mathematics, 197. Springer-Verlag, New York, 2000.

 
\bibitem[Go1]{Go1}
W. M. Goldman,
\emph{ The symplectic nature of fundamental groups of surfaces.}
Adv. in Math. 54 (1984), no. 2, 200--225.

\bibitem[Go2]{Go2}
W. M. Goldman,
\emph{Invariant functions on Lie groups and Hamiltonian flows of surface group representations.}
Invent. Math. 85 (1986), no. 2, 263--302.

 
 

 

\bibitem[KK1]{KK1}
N. Kawazumi, Y. Kuno,
\emph{The logarithms of Dehn twists.}
 Quantum Topol. 5 (2014), no. 3, 347–423.

 
\bibitem[KK2]{KK2}
N. Kawazumi, Y. Kuno,
\emph{Intersection of curves on surfaces and their applications to mapping class groups.}
Ann. Inst. Fourier 65 (2015), no. 6, 2711–2762. 

\bibitem[LBP]{Pro} L. Le Bruyn, C. Procesi, \emph{Semisimple representations of quivers.}
Trans. Amer. Math. Soc. 317 (1990), no. 2, 585–598. 

  

%\bibitem[MT1]{MTold}
%G. Massuyeau, V. Turaev,
%\emph{Fox pairings and generalized Dehn twists.}
 %Ann. Inst. Fourier (Grenoble) 63 (2013), no. 6, 2403–2456. 

\bibitem[MT]{MTnew}
G. Massuyeau, V. Turaev,
\emph{Quasi-Poisson structures on representation spaces of surfaces.}
International Math.  Research Notices (2014), 1–64.



%T. Schedler,
%\emph{A Hopf algebra quantizing a necklace Lie algebra canonically associated to a quiver.}
%Int. Math. Res. Not. 12 (2005), 725--760.

%\bibitem[Tu1]{Tu1}
%V.  Turaev,
%\emph{Intersections of loops in two-dimensional manifolds.}
%(Russian) Mat. Sb. 106(148) (1978),   566--588.
%English translation: Math. USSR, Sb. 35 (1979), 229--250.

%\bibitem[Tu2]{Tu2}
%V.  Turaev,
%\emph{Skein quantization of Poisson algebras of loops on surfaces.}
%Ann. Sci. \'Ecole Norm. Sup. (4) 24 (1991), no. 6, 635--704.


\bibitem[Tu1]{Tu3}
V.  Turaev,
\emph{Loops in surfaces and star-fillings.} 
arXiv:1910.01602.

\bibitem[Tu2]{Tu4}
V.  Turaev,
\emph{Quasi-Lie bialgebras of loops in quasi-surfaces.}
arXiv:2001.02279.

\bibitem[VdB]{VdB}
M. Van den Bergh,
\emph{Double Poisson algebras.}
Trans. Amer. Math. Soc. 360 (2008), no. 11, 5711--5769.

%\bibitem[Wo]{Wo}
%S. Wolpert,
%\emph{On the symplectic geometry of deformations of a
%hyperbolic surface.} Ann. of Math. (2) 117 (1983), 207--234.





                     \end{thebibliography}
    \end{document}